\documentclass[aap,MSNbibl,citesort,dvips]{arximspdf}
\usepackage{graphicx}

\doi{10.1214/10-AAP718}
\volume{21}
\issue{3}
\pubyear{2011}
\firstpage{1053}
\lastpage{1072}

\makeatletter

\newtheorem{theorem}{Theorem}

\newtheorem{lemma}[theorem]{Lemma}
\newtheorem{proposition}[theorem]{Proposition}
\newtheorem{corollary}[theorem]{Corollary}

\newproclaim{remark}{Remark}

\newcommand{\eps}{\varepsilon}

\newcommand{\R}{\mathbb{R}}
\newcommand{\N}{\mathbb{N}}
\newcommand{\cP}{{\mathcal P}}
\newcommand{\cN}{{\mathcal N}}
\newcommand{\cH}{{\mathcal H}}

\makeatother

\begin{document}
\begin{frontmatter}

\title{Hamilton cycles in random geometric graphs}
\runtitle{Hamilton cycles in random geometric graphs}

\begin{aug}
\author[A]{\fnms{J{\'o}zsef} \snm{Balogh}\thanksref{t1}\ead[label=e1]{jobal@math.uiuc.edu}},
\author[B]{\fnms{B\'ela} \snm{Bollob\'as}\thanksref{t2}\ead[label=e2]{b.bollobas@dpmms.cam.ac.uk}},
\author[C]{\fnms{Michael} \snm{Krivelevich}\thanksref{t3}\ead[label=e3]{krivelev@post.tau.ac.il}},
\author[D]{\fnms{Tobias} \snm{M\"uller}\thanksref{t4}\ead[label=e4]{tobias@cwi.nl}} and
\author[E]{\fnms{Mark} \snm{Walters}\corref{}\ead[label=e5]{M.Walters@qmul.ac.uk}}
\runauthor{J. Balogh et al.}
\affiliation{University of California and University of Illinois,
University~of~Cambridge~and University of Memphis, Tel Aviv
University, Centrum~voor~Wiskunde~en~Informatica, and Queen Mary, University~of~London}
\address[A]{J. Balogh\\
Department of Mathematics\\
University of Illinois\\
Urbana, Illinois 61801\\
USA\\
\printead{e1}}
\address[B]{B. Bollob\'as\\
Department of Pure Mathematics\\
\quad and Mathematical Statistics\\
University of Cambridge\\
Cambridge, CB3 0WB\\
United Kingdom\\
\printead{e2}}
\address[C]{M. Krivelevich\\
School of Mathematical Sciences\\
Tel Aviv University\\
Ramat Aviv 69978\\
Israel\\
\printead{e3}}
\address[D]{T. M\"uller\\
Centrum voor Wiskunde\hspace*{35pt}\\
\quad en Informatica\\
P.O. Box 94079\\
1090 GB Amsterdam\\
The Netherlands\\
\printead{e4}}
\address[E]{M. Walters\\
School of Mathematical Sciences\\
Queen Mary, University of London\\
London, E1 4NS\\
United Kingdom\\
\printead{e5}}
\end{aug}

\thankstext{t1}{This material is based upon work supported by NSF
CAREER Grant DMS-07-45185 and DMS-06-00303, UIUC Campus Research Board
Grants 09072 and 08086 and OTKA Grant K76099.}
\thankstext{t2}{Supported in part by NSF Grants DMS-05-05550,
CNS-0721983 and CCF-0728928 and ARO Grant W911NF-06-1-0076.}
\thankstext{t3}{Supported in part by USA-Israel BSF Grant
2006322, by Grant 1063/08 from the Israel Science Foundation and by
a Pazy memorial award.}
\thankstext{t4}{Supported in part by a VENI grant from Netherlands
Organisation for Research (NWO). The results in this paper are based on
work done while at Tel Aviv University, partially supported through an
ERC advanced grant.}

% HISTORY:
\received{\smonth{4} \syear{2009}}
\revised{\smonth{1} \syear{2010}}

% ABSTRACT
%
\begin{abstract}
We prove that, in the Gilbert model for a random geometric graph,
almost every graph becomes Hamiltonian exactly when it first becomes
2-connected. This answers a question of Penrose.

We also show that in the $k$-nearest neighbor model, there is a
constant $\kappa$ such that almost every $\kappa$-connected graph has a
Hamilton cycle.
\end{abstract}

% KEYWORDS
%

\end{frontmatter}

%s1 ###
\section{Introduction}\label{intro}
In this paper we mainly consider one of the frequently studied models
for random geometric graphs, namely the Gilbert model. Suppose that
$S_n$ is a $\sqrt n\times\sqrt n$ box and that $\cP$ is a Poisson
process in it with density 1. The points of the process form the
vertex set of our graph. There is a parameter $r$ governing the edges:
two points are joined if their (Euclidean) distance is at most $r$.

Having formed this graph we can ask whether it has any of the standard
graph properties, such as connectedness. As usual, we shall only
consider these for large values of $n$. More formally, we say that
$G=G_{n,r}$ has a property \textit{with high probability} (abbreviated to
whp) if the probability that $G$ has this property tends to one as
$n$ tends to infinity.

Penrose~\cite{Pen} proved that the threshold for connectivity is $\pi
r^2=\log n$. In fact he proved the
following very sharp result: suppose $\pi r^2=\log n +\alpha$ for some
constant~$\alpha$. Then the probability that $G_{n,r}$ is connected
tends to $e^{-e^{-\alpha}}$.

He also\vspace*{1pt} generalized this result to find the threshold for
$\kappa$-connectivity for \mbox{$\kappa\ge2$}: namely $\pi r^2 =\log n
+(2\kappa-3)\log\log n$. [Since the reader may be surprised that this
formula does not work for $\kappa=1$ we remark that this is due to
boundary effects: the threshold for $\kappa$-connectivity is the
maximum of two quantities: $\log n+(\kappa-1)\log\log n$ to
$\kappa$-connect the central points and $\log n+(2\kappa-3)\log\log n$
to \mbox{$\kappa$-connect} the points near the boundary. If one worked on the
torus instead of the square, then these boundary effects would
disappear.]

Moreover, he found the ``obstruction'' to
$\kappa$-connectivity. Suppose we fix the vertex set (i.e., the point
set in $S_n$) and ``grow'' $r$. This gradually adds edges to the
graph. For a monotone graph property $P$ let $\cH(P)$ denote the
smallest $r$ for which the graph on this point set has the property
$P$. Penrose showed that
\[
\cH\bigl(\delta(G)\ge\kappa\bigr)=\cH\bigl(\mathrm{connectivity}(G)\ge\kappa\bigr)
\]
whp: that is, as soon as the graph has minimum degree $\kappa$ it is
$\kappa$-connected whp.

He also considered the threshold for $G$ to have a Hamilton cycle.
Obviously a necessary condition is that the graph is 2-connected. In
the normal (Erd\H{o}s--R\'enyi) random graph this is also a sufficient
condition in the following strong sense. If we add edges to the graph
one at a time, then the graph becomes Hamiltonian exactly when it
becomes 2-connected (see
\cite{BB-Hamilton,Korshunov,KS} and~\cite{Posa}).

Penrose asked whether the same is true for a random geometric
graph. In this paper we prove the following theorem answering this
question.
\begin{theorem}\label{t:2d-gilbert}
Suppose that $G=G_{n,r}$ is the two-dimensional Gilbert model. Then
\[
\cH(\mbox{$G$ is 2-connected})=\cH(\mbox{$G$ has a Hamilton cycle})
\]
whp.
\end{theorem}

Combining this with Penrose's results mentioned above we see
that, if $\pi r^2=\log n+\log\log n+\alpha$, then the\vspace*{1pt} probability that
$G$ has a Hamilton cycle tends to $e^{-e^{-\alpha}-\sqrt\pi
e^{-\alpha/2}}$~(the second term in the exponent is the contribution
from points near the boundary of the square).

Some partial progress has been made on this question previously.
Petit~\cite{Petit} showed that if $\pi r^2/\log n$ tends to infinity,
then $G$ is, whp, Hamiltonian, and D\'iaz, Mitsche and
P\'erez~\cite{DMP} proved that if $\pi r^2>(1+\eps)\log n$ for some
$\eps>0$ then $G$ is Hamiltonian whp. (Obviously, $G$ is not
Hamiltonian if $\pi r^2<\log n$ since whp $G$ is not connected!)
Finally using a similar method to~\cite{DMP} together with significant
case analysis, Balogh, Kaul and Martin~\cite{BKM} proved for the
special case of the $\ell_\infty$ norm in two dimensions
that the graph does become Hamiltonian exactly when it becomes 2-connected.

Our proof generalizes to higher dimensions and to other norms.
The Gilbert model makes sense with any norm and in any number of
dimensions: we let $S_n^d$ be the $d$-dimensional hypercube with
volume $n$. We prove the analog of Theorem~\ref{t:2d-gilbert} in
this setting.
\begin{theorem}\label{t:multi-dim-gilbert}
Suppose that the dimension $d\ge2$ and \mbox{$\|\cdot\|$}, a~$p$-norm for some $1\le p\le\infty$, are fixed. Let $G=G_{n,r}$ be
the resulting Gilbert model. Then
\[
\cH(\mbox{$G$ is 2-connected})=\cH(\mbox{$G$ has a Hamilton cycle})
\]
whp.
\end{theorem}

The proof is very similar to that of
Theorem~\ref{t:2d-gilbert}. However, there are some significant extra
technicalities.

To give an idea why these occur consider connectivity in the Gilbert
model in the cube $S_n^3$ (with the Euclidean norm). Let $A$ be the
volume of a sphere of radius~$r$. We count the expected number of
isolated points in the process which are away from the boundary of
the cube. The probability a point is isolated is $e^{-A}$, so the
expected number of such points is $ne^{-A}$, so the threshold for the
existence of a central isolated point is about $A=\log n$.

However, consider the probability that a point near a face of the cube
is isolated: there are approximately $ n^{2/3}$ such points, and the
probability that they are isolated is about $e^{-A/2}$ (since about
half of the sphere about the point is outside the
cube~$S_n^3$). Hence,\vspace*{-1pt} the expected number of such points is
$n^{2/3}e^{-A/2}$, so the threshold for the existence of an isolated
point near a face is about $A=\frac43 \log n$. In other words
isolated points are much more likely near the boundary. These boundary
effects are the reason for many of the extra technicalities.

We remark that Theorem~\ref{t:multi-dim-gilbert} is trivially true for
$d=1$: indeed, if $G$ is 2-connected then there are two vertex
disjoint paths from the left-most vertex to the right-most vertex. By
adding any remaining vertices to one of these paths these two paths
form a Hamilton cycle.
%{\it The $k$-nearest neighbour model.}

\subsection*{The $k$-nearest neighbor model}

We also consider a second model for random geometric graphs: namely
the $k$-nearest neighbor graph. In this model the initial setup is
the same as in the Gilbert model: the vertices are given
by a Poisson process of density one in the square $S_n$, but this time
each vertex is joined to its $k$ nearest neighbors (in the Euclidean
metric) in the box. This naturally gives rise to a $k$-regular
directed graph,\vadjust{\goodbreak} but we form a simple graph $G=G_{n,k}$ by ignoring the
direction of all the edges. It is easily checked that this gives us a
graph with degrees between $k$ and $6k$.

Xue and Kumar~\cite{XuKu} showed that there are constants $c_1,c_2$
such that if $k<c_1\log n$, then the graph $G_{n,k}$ is, whp, not
connected, and that if $k>c_2\log n$ then $G_{n,k}$ is, whp,
connected. Balister et al.~\cite{BBSW1}
proved reasonably good bounds on the constants: namely $c_1=0.3043$
and $c_2=0.5139$, and later~\cite{BBSW3} proved that there is some
critical constant $c$ such that if $k=c'\log n$ for $c'<c$, then the
graph is disconnected whp, and if $k=c'\log n$ for $c'>c$, then it is
connected whp. Moreover, in~\cite{BBSW2}, they showed that in the
latter case the graph is $s$-connected whp for any fixed $s\in\N$.

We would like to prove a sharp result like the above; that is, that as
soon as the graph is 2-connected it has a Hamilton cycle. However, we
prove only the weaker statement that some (finite) amount of
connectivity is sufficient. Explicitly, we show the following.
\begin{theorem}\label{t:knear}
Suppose that $k=k(n)$, that $G=G_{n,k}$ is the two-dimensional
$k$-nearest neighbor graph (with the\vspace*{1pt} Euclidean norm) and that $G$
is $\kappa$-connected for $\kappa=5\cdot10^7$ whp. Then $G$ has a
Hamilton cycle whp.
\end{theorem}

Analogous results could be proved in higher
dimensions and for other norms but we do not do so here.

\subsection*{Binomial point process}
To conclude this section we briefly mention a closely related model:
instead of choosing the points in $S_n$ according to a Poisson process
of density one we choose $n$ points uniformly at random, and then form
the corresponding graph. This new model is very closely related to our
first model (the Gilbert model). Indeed, Penrose originally proved his
results for the Binomial Point Process but it is easy to check that
this implies them for the Poisson Process.

It is very easy to modify our proof to this new model. Indeed, in
very broad terms each of our arguments consists of two steps: first we
have an essentially trivial lemma that says the random points are
``reasonably'' distributed, and then we have an argument saying that
if the points are reasonably distributed and the resulting graph is
two-connected then the resulting graph necessarily has a Hamilton
cycle. The second of these steps is entirely deterministic, so only the
essentially trivial lemma needs modifying.

%s2 ###
\section{\texorpdfstring{Proof of Theorem \protect\ref{t:2d-gilbert}}{Proof of Theorem 1}}\label{s:gilbert-2d}

We divide the proof into five parts: first we tile the square $S_n$
with small squares in a standard tessellation argument. Second we
identify ``difficult'' subsquares. Roughly, these will be squares
containing only a few points, or squares surrounded by squares
containing only a few points. Third we prove some lemmas about the
structure of the difficult subsquares. In stage 4 we deal with the
difficult subsquares. Finally we use the remaining easy subsquares to
join everything together.\vadjust{\goodbreak}

\subsection*{Stage 1: Tessellation}

Let $r_0=\sqrt{(\log n)/\pi}$ (so $\pi r_0^2=\log n$), and let $r$ be
the random variable $\cH(G$ is 2-connected$)$. Let
$s=r_0/c=c'\sqrt{\log n}$ where $c$ is a large constant to be chosen
later (1000 will do). We tessellate the box $S_n$ with small squares
of side length $s$. Whenever we talk about distances between squares
we will always be referring to the distance between their
centers. Moreover, we will divide all distances between squares by
$s$, so, for example, a~square's four nearest neighbors all have
distance one.

By Penrose's result~\cite{Pen2} mentioned in the \hyperref[intro]{Introduction} we
may assume that $(1-1/2c)r_0<r<(1+1/2c)r_0$: formally the collection
of point sets which do not satisfy this has measure tending to zero as
$n$ tends to infinity, and we ignore this set.

Hence points in squares at distance $\frac{r-\sqrt2s}{s} \ge
\frac{r_0-2s}{s}=c-2$ are always joined, and points in squares at
distance $\frac{r+\sqrt2 s}s\le\frac{r_0+2s}s=c+2$ are never
joined.

\subsection*{Stage 2: The ``difficult'' subsquares}

We call a square \textit{full} if it contains at least $M$ points for
some $M$ to be determined later ($10^7$ will do), and \textit{nonfull}
otherwise. Let $N_0$ be the set of nonfull squares. We say two
nonfull squares are joined if their $\ell_\infty$ distance is at most
$4c-1$ and define $\cN$ to be the collection of nonfull components.

First we bound the size of the largest component of nonfull squares
(here, and throughout this paper, we use size to refer to the number
of vertices in the component).
\begin{lemma}\label{l:non-full-size}
For any $M$, the largest component of nonfull squares in the above
tesselation has size at most
\[
U=\lceil\pi(c+2)^2
\rceil
\]
whp.

Also, the largest component of nonfull squares including a square
within $c$ of the boundary of $S_n$ has size at most $U/2$ whp.
Finally, there is no nonfull square within distance $Uc$ of a corner
whp.
\end{lemma}
\begin{pf}
We shall make use of the following simple result: suppose that $G$ is
any graph with maximal degree $\Delta$, and $v$ is a vertex in
$G$. Then the number of connected subsets of size $n$ of $G$
containing $v$ is at most $(e\Delta)^n$ (see, e.g., Problem~45 of
\cite{CTM}).

Hence, the number of potential components of size $U$ containing a
particular square is at most $(e(8c)^2)^{U}$ so, since there are less
than $n$ squares, the total number of such potential components is at
most $n(e(8c)^2)^{U}$. The probability that a square is nonfull is at
most $2s^{2M}e^{-s^2}/M!$. Hence, the expected number of components of
size at least $U$ is at most
\[
n\bigl(2s^{2M}e^{-s^2}(e(8c)^2)/M!\bigr)^{U} \le n\biggl(2(\log
n)^M\frac{e(8c)^2}{M!}\biggr)^{U}\exp\biggl(-\frac{(c+2)^2\log
n}{c^2}\biggr),\vadjust{\goodbreak}
\]
which tends to zero as $n$ tends to infinity; that is, whp, no such
component exists.

For the second part there are at most $4c\sqrt n$ squares within
distance $c$ of the boundary of $S_n$, and the result follows as above.

Finally, there are only $4U^2c^2$ squares within distance $Uc$ of a
corner. Since the probability that a square is nonfull tends to zero
we see that there is no such square whp.
\end{pf}

Note that this is true independently of $M$ which is important since
we will want to choose $M$ depending on $U$.

In the rest of the argument we shall assume that there is no nonfull
component of size greater than $U$, no nonfull component of size
$U/2$ within $c$ of an edge and no nonfull square within $Uc$ of a
corner.

Between these components of nonfull squares there are numerous full
squares. To define this more precisely let $\widehat G$ be the graph
with vertex set the small squares, and where each square is joined to
all others within $(c-2)$ of this square (in the Euclidean norm).
Since the probability a square is in $N_0$ (i.e., is nonfull) is
$1-o(1)$, the graph $\widehat G\setminus N_0$ has one giant component
consisting of almost all the squares. We call this component \textit{sea}. (We give an equivalent formal definition just before
Corollary~\ref{c:local-sea}.)

The idea is that it is trivial to find a cycle visiting every point of
the process in a square in the sea, and that we can
extend this cycle to a Hamilton cycle by adding each nonfull
component (and any full squares cut off by it) one at a time. However,
it is easier to phrase the argument by starting with the difficult
parts and then using the sea of full squares.

\subsection*{Stage 3: The structure of the difficult subsquares}

Consider one component $N\in\cN$ of the nonfull squares, and suppose
that it has size $u$. By Lemma~\ref{l:non-full-size} we know $u<U$. We
will also consider $N_{2c}$: the $2c$-blow-up of $N$: that is the set
of all squares with $\ell_\infty$ distance at most $2c$ from a square
in~$N$.

Now some full squares may be cut off from the rest of the full squares
by nonfull squares in $N$. More precisely the graph $\widehat
G\setminus N$ has one component $A=A(N)$
consisting of all but at most a bounded number of
squares (since we have removed at most $U$ squares from $\widehat G$).
We call $A^c$ the \textit{cutoff} squares.

%f1 ###
%
\begin{figure}

\includegraphics{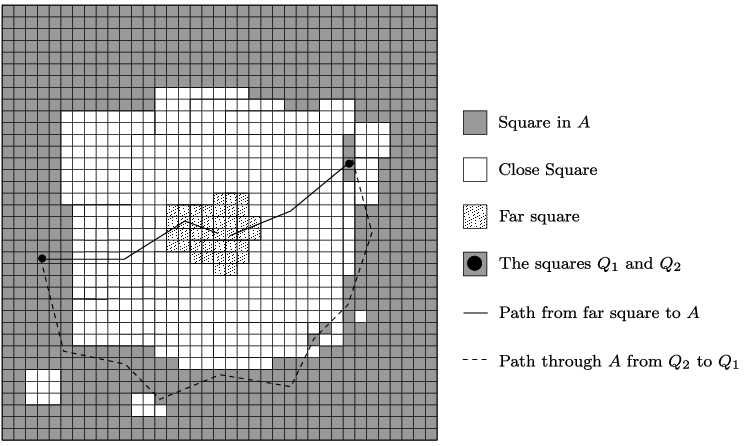}

\caption{A small part of $S_n$ containing the nonfull component $N$
and the corresponding set $A$, far squares and close squares. It
also shows the two vertex disjoint paths from the far squares to $A$
and the path joining $Q_2$ to $Q_1$ (see stage 4).}
\label{fig:far-close}
\end{figure}

We split the cutoff squares into two classes: those with a neighbor
in $A$ (in~$\widehat G$) which we think of as being ``close'' to $A$, and
the rest, which we shall call \textit{far} squares. All the close
squares must be in $N$ (since otherwise they would be part of~$A$).
However, we do not know anything about the far squares: they
may be full or nonfull. See Figure~\ref{fig:far-close} for a
picture.
\begin{lemma}\label{l:far}
No two far squares are more than $\ell_\infty$ distance $c/10$ apart.
\end{lemma}
\begin{remark*} This does not say whp since we are assuming this nonfull
component has size at most $U$.\vadjust{\goodbreak}
\end{remark*}
\begin{pf}%{Proof of Lemma~\ref{l:far}}
Suppose not.

Suppose, first, that no point of $N$ is within $c$ of the edge of
$S_n$, and that the two far squares are at horizontal distance at
least $c/10$. Then consider the left-most far square. All squares
which are to the left of this and with distance to this square less
than $(c-2)$ must be close and thus in $N$. Similarly with the
right-most far square. Also at least $(c-2)$ squares [in fact nearly
$2(c-2)$] in each of at least $c/10$ columns between the original two
far squares must be in $N$. This is a total of about $\pi(c-2)^2 +
(c-2)c/10>U$ which is a contradiction (provided we chose $c$
reasonably large).

If there is a point of $N$ within $c$ of the boundary, then the above
argument gives more than $U/2$ nonfull squares. Indeed, either it
gives half of each part of the above construction, or it gives all
of one end and all the side parts. This contradicts the second part
of our assumption about the size of nonfull components.

We do not need to consider a component near two sides: it cannot be
large enough to be near two sides. It also cannot go across a corner,
since no square within distance $Uc$ of a corner is nonfull.
\end{pf}

This result can also be deduced from a result
of Penrose, as we do in
the next section. We have the following instant corollary.
\begin{corollary}\label{c:far}
The graph $\widehat G$ restricted to the far squares is
complete.
\end{corollary}
\begin{corollary}\label{c:cutoff}
The set of cutoff squares $A^c$ is contained in $N_c$ (the $c$-blow-up
of $N$). In particular,
the set $\Gamma(A^c)$ of neighbors in $\widehat G$ of $A^c$ is
contained in~$N_{2c}$.\vspace*{-2pt}
\end{corollary}
\begin{pf}
Suppose $A^c\not\subseteq N_c$.
Let $x$ be a square in $A^c\setminus N_c$. First, $x$
cannot be a neighbor of any square in $A$ or $x$ would also be in
$A$; that is, $x$ is a far square.

Now, let $y$ be any square with $\ell_\infty$ distance $c/5$ from $x$.
The square $y$ cannot be~in $N$ since then $x$ would be in
$N_c$. Therefore, $y$ cannot be a neighbor of any square in $A$ since
then it would be in $A$ and, since $x$ and $y$ are joined in $\widehat
G$, $x$ would be in $A$; that is, $y$ is also a far square. Hence, $x$
and $y$ are both far squares with $\ell_\infty$ distance $c/5$ which
contradicts Lemma~\ref{l:far}.\vspace*{-2pt}~%
\end{pf}

In particular, Corollary~\ref{c:cutoff} tells us that the sets of
squares cutoff by different nonfull components and all their
neighbors are disjoint (obviously the $2c$-blow-ups are disjoint).

We now formally define the \textit{sea} $\widetilde A=\bigcap_{N\in
\cN}A(N)$. We show later (Corollary~\ref{c:sea}) that $\widetilde A$
is connected and, thus, that this is the same as our earlier informal
definition. The following corollary is immediate from
Corollary~\ref{c:cutoff}.\vspace*{-2pt}
\begin{corollary}\label{c:local-sea}
For any $N\in\cN$ we have $\widetilde A\cap N_{2c}=A(N)\cap
N_{2c}$.\vspace*{-2pt}
\end{corollary}

The final preparation we need is the following lemma.\vspace*{-2pt}
\begin{lemma}\label{l:connected-boundary}
The set $N_{2c}\cap A$ is connected in $\widehat G$.\vspace*{-2pt}
\end{lemma}

Since the proof will be using a standard graph theoretic result, it is
convenient to define one more graph $\widehat G_1$: again the vertex
set is the set of small squares, but this time each square is joined
only to its four nearest neighbors; that is, $\widehat G_1$ is the
ordinary square lattice. We need two quick definitions. First, for a
set $E\in\widehat G_1$ we define the \textit{boundary} $\partial_1E$ of
$E$ to be set of vertices in $E^c$ that are neighbors (in~$\widehat
G_1$) of a vertex in $E$. Second, we say a set $E$ in $\widehat
G_1$ is \textit{diagonally connected} if it is connected when we add the
edges between squares which are diagonally adjacent (i.e., at distance
$\sqrt2$) to $\widehat G$. The lemma we need is the following; since
its proof is short we include it here for completeness. (It is also an
easy consequence of the unicoherence of the square (see, e.g., page 177
of~\cite{Penbook}).)\vspace*{-2pt}
\begin{lemma}\label{l:1-boundary}
Suppose that $E$ is any subset of $\widehat G_1$ with $E$ and $E^c$
connected. Then $\partial_1 E$ is diagonally connected: in
particular, it is connected in~$\widehat G$.\vspace*{-2pt}
\end{lemma}
\begin{pf}
Let $F$ be the set of edges of $\widehat G_1$ from $E$ to $E^c$, and
let $F'$ be the corresponding set of edges in the dual
lattice. Consider the set $F'$ as a subgraph of the dual lattice. It
is easy to check that every vertex has even degree except vertices
on the boundary of $\widehat G_1$. Thus we can decompose $F'$ into\vadjust{\goodbreak}
pieces, each of which is either a cycle or a path starting and
finishing at the edge of $\widehat G_1$. Any such cycle splits
$\widehat G_1$ into two components, and we see that one of these
must be exactly $E$ and the other $E^c$. Thus $F'$ is a single
component in the dual lattice, and it is easy to check that implies
that $\partial_1E$ is diagonally connected.\vspace*{-1pt}
\end{pf}
\begin{pf*}{Proof of Lemma~\ref{l:connected-boundary}}
Consider $\widehat G_1\setminus N_{2c}$. This splits into components
$B_1,B_2,\ldots, B_m$. By definition each $B_i$ is connected. Moreover,
each $B_i^c$ is also connected. Indeed, suppose $x,y\in B_i^c$. Then
there is an $xy$ path in $\widehat G_1$. If this is contained in $B_i^c$
we are done. If not then it must meet $N_{2c}$, but $N_{2c}$ is
connected. Hence we can take this path until it first meets $N_{2c}$,
go through $N_{2c}$ to the point where the path last leaves $N_{2c}$
and follow the path on to $y$. This gives a path in $B_i^c$.

Hence, by Lemma~\ref{l:1-boundary}, we see that each $\partial_1 B_i$
is connected in $\widehat G$ for each $i$ (where $\partial_1$ denotes the
boundary in $\widehat G_1$). Obviously $\partial_1 B_i\subset N_{2c}$.

As usual, for a set of vertices $V$ let $\widehat G[V]$ denote the
graph $\widehat G$ restricted to the vertices in $V$.\vspace*{-1pt}
\begin{claim*}
Any two vertices in $\bigcup_{i=1}^m \partial_1 B_i$ are connected in
$\widehat G[A\cap N_{2c}]$.\vspace*{-1pt}
\end{claim*}
\begin{pf}
Suppose not. Without loss of generality assume that, for some $k<m$,
$\widehat G[\bigcup_{i=1}^k \partial_1B_i]$ is connected and that no
other $\partial_1B_i$ is connected via a path to it. Pick $x\in B_1$
and $y\in B_m$. Both $x$ and $y$ are in $A$ (since they are not in
$N_{2c}$ and $A^c\subset N_{2c}$ by Corollary~\ref{c:cutoff}).

Hence there is a path from $x$ to $y$ in $A$. Consider the last time
it leaves $\bigcup_{i=1}^k B_i$. The path then moves around in $N_{2c}$
before entering some $B_j$ with $j>k$. This gives rise to a path in
$A\cap N_{2c}$ from a point in $\bigcup_{i=1}^k\partial_1 B_i$ to a point
in $\partial_1 B_j$, contradicting the choice of
$k$.\vspace*{-1pt}
\end{pf}

We now complete the proof of Lemma~\ref{l:connected-boundary}. To
avoid clutter we shall say that two points are \textit{joined} if they
are connected by a path. Suppose that $x,y\in A\cap N_{2c}$. Since
$A$ is connected there is a path in $A$ from $x$ to $y$. If the path
is contained in $N_{2c}$ we are done. If not, consider the first time
the path leaves $N_{2c}$. It must enter one of the $B_i$, crossing the
boundary $\partial_1B_i$. Hence $x$ is joined to some $w\in
\partial_1B_i$ in $A\cap N_{2c}$. Similarly, by considering the last
time the path is not in $N_{2c}$ we see that $y$ is joined to some
$z\in\partial_1B_j$ for some~$j$. However, since the claim showed that
$w$ and $z$ are joined in $A\cap N_{2c}$, we see that $x$ and $y$ are
joined in $A\cap N_{2c}$.\vspace*{-1pt}
\end{pf*}
\begin{corollary}\label{c:sea}
The set of sea squares $\widetilde A$ is connected in $\widehat G$.\vspace*{-1pt}
\end{corollary}
\begin{pf}
Given two squares $x,y$ in $\widetilde A$, pick a path in $\widehat G$
from $x$ to $y$. Now for each nonfull component $N$ in turn do the
following. If the path misses $N_{2c}$ do nothing. Otherwise let $w$
be the first point on the path in $N_{2c}$ and $z$ be the last point
in $N_{2c}$. Replace the $xy$ path by the path $xw$, any path $wz$ in
$A(N)\cap N_{2c}$ and then the path $zy$.\vadjust{\goodbreak}

At each stage the modification ensured that the path now lies in
$A(N)$. Also, the only vertices added to the path are in $N_{2c}$
which is disjoint from all the previous~$N'_{2c}$, and thus from all
previous sets $A(N')$. Hence, when we have done this for all nonfull
components the path lies in every $A(N')$, that is, in $\widetilde
A$. Hence, $\widetilde A$ is connected.
\end{pf}

\subsection*{Stage 4: Dealing with the difficult subsquares}

We deal with each nonfull component $N\in\cN$ in turn. Fix one such
component $N$.

Let us deal with the far squares first. There are three possibilities:
the far squares contain no points at all, they contain one point in
total or they contain more than one point. In the first case, do
nothing and proceed to the next part of the argument.

In the second case, by the 2-connectivity of $G$, we can find two
vertex disjoint paths from this single vertex $v_1$ to points in
squares in $A$. In the third case pick two points $v_1$ and $v_2$ in
the far squares. Again by 2-connectivity we can find vertex disjoint
paths from these two vertices to points in squares in $A$.

Suppose that the path from $v_1$ meets $A$ in square $Q_1$ at point
$q_1$ and the other path (either from $v_2$ or the other path from
$v_1$ again) meets $A$ in square $Q_2$ at point~$q_2$. Let $P_1,P_2$
be the squares containing the previous points on these paths. Since
no two points in squares at (Euclidean) distance $(c+2)$ are joined we
see that $P_1$ is within $(c+2)$ of $Q_1$. Since $P_1\notin A$ we
have that some square on a shortest $P_1Q_1$ path in $\widehat G_1$ is
in $N$ and thus that $Q_1\in N_{2c}$. Similarly $Q_2\in
N_{2c}$. Combining we see that both $Q_1$ and $Q_2$ are in $N_{2c}\cap
A$. By Lemma~\ref{l:connected-boundary}, we know that $N_{2c}\cap A$
is connected in $\widehat G$ so we can find a path from $Q_1$ to $Q_2$
in $N_{2c}\cap A$ in~$\widehat G$. This ``lifts'' to a path in $G$
going from $q_2$ to a point other than $q_1$ in $Q_1$ using at most
one vertex in each subsquare on the way and never leaving $N_{2c}$.

Construct a path starting and finishing in $Q_1$ by joining together the
following paths:
\begin{enumerate}
\item the path from $q_1$ to $v_1$;
\item a path starting at $v_1$ going round all points in the far
region (except any such points on the $q_1v_1$ or $q_2v_2$ paths)
finishing back at $v_2$. (Corollary~\ref{c:far} guarantees the
existence of such a path.) We omit this piece if there is just one
far vertex;
\item the path $v_2$ to $q_2$;
\item the path from $q_2$ through the sea back to $Q_1$ constructed
above.
\end{enumerate}

Since $Q_1\in A\cap N_{2c}$, by Corollary~\ref{c:local-sea} we have
that $Q_1\in\widetilde A$. Combining, we have a path starting and
finishing in the same subsquare of the sea $\widetilde A$ (i.e.,
$Q_1$) containing all the vertices in the far region.

Next we deal with the close squares: we deal with each close square
$P$ in turn. Since $P$ is a close square we can pick $Q\in A$
with
$PQ$ joined in~$\widehat G$. In the following we ignore all points
that we have used in the path constructed above and any points
already used when dealing with other close squares.

If the square $P$ has no point in it we ignore it. If it has one point
in it, then join that point to two points in $Q$.

If it has two or more points in it then pick two of them $x,y$: and
pick two points $uv$ in $Q$ (we choose $M$ large enough to ensure that
we can find these two unused points in $Q$, see below).
Place the path formed by the edge $ux$ round all the remaining unused
vertices in the cutoff square finishing at $y$ and back to the square
$Q$ with the edge $yv$ in the cycle we are constructing.

The square $Q$ is a neighbor of $P\in A^c$ so, by
Corollary~\ref{c:cutoff} is in $N_{2c}$. Since $Q$ is also in $A$ we
see, by Corollary~\ref{c:local-sea} as above, that $Q\in\widetilde A$.

When we have completed this construction we have placed every vertex
in a cutoff square on one of a collection of paths, each of which
starts and finishes at the same square in the sea (although different
paths may start and finish in different squares in the sea).

We use at most $2U+2$ vertices from any square in $A=A(N)$ when doing
this, so, provided that $M>2U+2+(2c+1)^2$, there are at least
$(2c+1)^2$ unused vertices in each square of $A$ when we finish this.
Moreover, obviously the only squares touched by this construction are
in $N_{2c}$, and for distinct nonfull components these are all
disjoint. Hence, when we have done this for every nonfull component
$N\in\cN$ there are at least $(2c+1)^2$ unused vertices in each square
of the sea $\widetilde A$.

\subsection*{Stage 5: Using the subsquares in the sea to join everything
together}

It just remains to string everything together. This is easy. Since, by
Corollary~\ref{c:sea}, the sea of squares $\widetilde A$
is connected, there is a spanning tree for $\widetilde A$. By doubling
each edge we can think of this as a cycle, as in
Figure~\ref{fig:tree-gilbert}. This cycle visits each square at most
%
%f2 ###
%
\begin{figure}[b]

\includegraphics{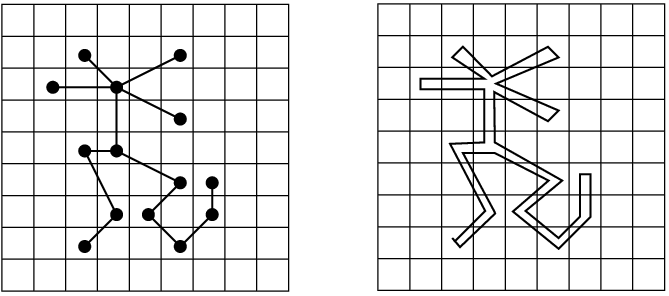}

\caption{A tree of subsquares and its corresponding tree cycle.}
\label{fig:tree-gilbert}
\end{figure}
$(2c+1)^2$ times. (In fact, by choosing a spanning tree such that the\vadjust{\goodbreak}
sum of the edge lengths is minimal we could assume that it visits each
vertex at most six times but we do not need this.) Convert this into a
Hamilton cycle as follows. Start at an unused vertex in a square of
the sea. Move to any (unused) vertex in the next square in the tree
cycle. Then, if this is the last time the tree cycle visits this
square, visit all remaining vertices and join in all the paths
constructed in the first part of the argument, then leave to the next
square in the tree cycle. If it is not the last time the tree cycle
visits this square, then move to any unused vertex in the next square
in the tree cycle. Repeat until we complete the tree cycle. Then join
in any unused vertices and paths to this square constructed earlier
before closing the cycle.

%s3 ###
\section{Higher dimensions}
We generalise the proof in the previous section to higher dimensions
and any $p$-norm. Much of the argument is the same, in particular,
essentially all of stages four and five. We include details of all
differences but refer the reader to the previous section where the
proof is identical.

\subsection*{Stage 1: Tessellation}

We work in the $d$-dimensional hypercube $S_n^d$ of volume~$n$ (for
simplicity we will abbreviate hypercube to cube in the following). As
mentioned in the \hyperref[intro]{Introduction}, we no longer have a nice formula for
the critical radius: the boundary effects dominate.

Instead, we consider the expected number of isolated vertices
$E=E(r)$. We need a little notation: let $A_r$ denote the set
$\{x\in S^d_n\dvtx d(x,A)\le r\}$ and \mbox{$|\cdot|$} denote Lebesgue measure.

We have $E=\int_{S_n^d}\exp(-|\{x\}_r|)\,dx$. Let $r_0=r_0(n)$ be such
that $E(r_0)=1$. As before fix $c$ a large constant to be determined
later, and let $s=r_0/c$. It is easy to see that $r_0^d=\Theta(\log
n)$ and $s^d=\Theta(\log n)$. We tile the cube $S_n^d$ with small
cubes of side length $s$.

As before, let $r=\cH(G$ is 2-connected$)$. By Penrose (Theorems 1.1
and~1.2 of~\cite{Pen2} or Theorems 8.4 and 13.17 of~\cite{Penbook})
the probability that $r\notin[r_0(1-1/2c),r_0(1+1/2c)]$ tends to
zero and we ignore all these point sets. (Note that these two of
Penrose's results are not claimed for $p=1$. However, since for any
$\eps>0$ we
can pick $p>1$ such that $B_1(r)\subset B_{p}(r)\subset
B_1((1+\eps)r)$ [where $B_1(r)$ and $B_{p}$ denote the $l_1$ and
$l_{p}$ balls of radius $r$, resp.], the above bound on
$r$ for $p=1$ follows from Penrose's results for $p>1$.)

This time any two points in cubes at distance
$\frac{r-s\sqrt{d}}{s}\ge\frac{r_0-ds}s=c-d$ are joined, and no
points in cubes at distance $\frac{r+s\sqrt{d}}s\le
\frac{r_0+ds}s=c+d$ are joined.

\subsection*{Stage 2: The ``difficult'' subcubes}

Exactly as before we define nonfull cubes to be those containing
at most $M$ points, and we say two are joined if they have $\ell
_\infty$
distance at most $4c-1$.

We wish to prove a version of Lemma~\ref{l:non-full-size}. However, we
have several possible boundaries: for example, in three dimensions we
have the center, the faces, the edges and the corners. We call a\vadjust{\goodbreak}
nonfull component containing a cube $Q$ \textit{bad} if it consists of
at least $(1+1/c)|Q_{r_0}|/s^d$ cubes. (Note a component can be bad
for some cubes and not others.)
\begin{lemma}
The expected number of bad components tends to zero as $n$ tends to
infinity. In particular there are no bad components whp.
\end{lemma}
\begin{pf}
The number of connected sets of size $U$ containing a
particular cube is at most $(e(8c)^d)^{U}$. The probability that a
cube is nonfull is at most $2s^{dM}e^{-s^d}/M!$. Since
$\min\{|Q_{r_0}|\dvtx\mathrm{cubes}\ Q\}=\Theta(\log n)$ and
$s^d=\Theta(\log n)$, the expected number of bad components is at most
\begin{eqnarray*}
&&\sum_{\mathrm{cubes}\
Q}\bigl(2s^{dM}e^{-s^d}(e(8c)^d)/M!\bigr)^{(1+1/c)|Q_{r_0}|/s^d}\\[-2pt]
&&\qquad=\sum_{\mathrm{cubes}\
Q}\bigl(2s^{dM}(e(8c)^d)/M!\bigr)^{(1+1/c)|Q_{r_0}|/s^d}\exp\bigl(-(1+1/c)|Q_{r_0}|\bigr)\\[-2pt]
&&\qquad=o(1)\sum_{\mathrm{cubes}\ Q}\exp(-|Q_{r_0}|)\\[-2pt]
&&\qquad\le o(1)\int_{S_n^d}\exp(-|\{x\}_{r_0}|)\,dx\\[-2pt]
&&\qquad=o(1) E(r_0)\\
&&\qquad=o(1).
\end{eqnarray*}
\upqed\end{pf}

(Again, note that this is true independently of $M$.)

From now on we assume that there is no bad component.

\subsection*{Stage 3: The structure of the difficult subcubes}

In this stage we will need one extra geometric result of Penrose, a~case of Proposition 5.15 of~\cite{Penbook} (see also Proposition 2.1
of~\cite{Pen2}).
\begin{proposition}\label{p:penrose}
Suppose $d$ is fixed and that \mbox{$\|\cdot\|$} is a $p$-norm for some $1\le
p\le\infty$. Then there exists $\eta>0$ such that if $F\subset O^d$
(the positive orthant in~$\R^d$) is compact with $\ell_\infty$
diameter at least $r/10$, and $x$ is a point of $F$ with minimal $l_1$
norm; then $|F_r|\ge|F|+|\{x\}_r|+\eta r^d$.
\end{proposition}

We begin this stage by proving Lemma~\ref{l:far} for this model.
\begin{lemma}\label{l:far-high-dim}
No two far cubes are more than $\ell_\infty$ distance $c/10$ apart.
\end{lemma}
\begin{pf}
Suppose not. Then let $F$ be the set of far cubes, let $x$ be a
point of $F$ closest to a corner in the $l_1$ norm and let $Q$ be
the cube containing $x$ (or any of the possibilities if it is on the
boundary between cubes). We know that all the cubes
within $(c-d)$ of a far cube are not in $A$. Hence all such cubes
which are not far must be close, and thus nonfull.\vadjust{\goodbreak}

The number of close cubes is at least
\begin{eqnarray*}
\frac{|F_{(c-2d)s}\setminus F|}{s^d}
&\ge&\frac{|\{x\}_{(c-2d)s}|+\eta((c-2d)s)^d}{s^d} \qquad\mbox{by
Proposition~\ref{p:penrose}}\\
&\ge&\frac{|Q_{(c-3d)s}|+\eta r_0^d/2}{s^d}\qquad\mbox{provided $c$ is large
enough}\\
&=&\frac{|Q_{(1-3d/c)r_0}|+\eta r_0^d/2}{s^d}\\
&\ge&\frac{(1-3d/c)^d|Q_{r_0}|+\eta r^d_0/2}{s^d}\\
&>&\frac{(1+1/c)|Q_{r_0}|}{s^d}\qquad \mbox{provided $c$ is large enough.}
\end{eqnarray*}
This shows that the component
is bad which is a contradiction.
\end{pf}

Corollaries~\ref{c:far},~\ref{c:cutoff} and~\ref{c:local-sea} hold
exactly as before. Lemma~\ref{l:connected-boundary} also holds, we
just need to replace Lemma~\ref{l:1-boundary} by the following
higher-dimensional analogue. Note that, even in higher dimensions we
say two
squares are diagonally connected if their centers have distance
$\sqrt{2}$.
\begin{lemma}\label{l:1-boundary-hd}
Suppose that $E$ is any subset of $\widehat G_1$ with $E$ and $E^c$
connected. Then $\partial_1 E$ is diagonally connected: in
particular, it is connected in~$\widehat G$.
\end{lemma}
\begin{remark*}
Again the final conclusion of connectivity in $\widehat G$ is
an easy consequence of unicoherence, this time of the
hypercube.
\end{remark*}
\begin{pf}%{Proof of Lemma~\ref{l:1-boundary-hd}}
Let $I$ be a (diagonally connected) component of $\partial_1 E$. We
aim to show the $I=\partial_1E$ and, thus, that $\partial_1 E$ is
diagonally connected.
\begin{claim*}
Suppose that $C$ is any circuit in $\widehat G_1$. Then the number of
edges of $C$ with one end in $E$ and the other end in $I$ is even.
\end{claim*}
\begin{pf}
We say that a circuit is contractible to a single point using the
following operations. First, we can remove an out and back
edge. Second, we can do the following two-dimensional move. Suppose
that two consecutive edges of the circuit form two sides of a square;
then we can replace them by the other two sides of the square keeping
the rest of the circuit the same. For example, we can replace
$(x,y+1,\vec z)\to(x+1,y+1, \vec z)\to(x+1,y,\vec z)$ in the circuit
by $(x,y+1,\vec z)\to(x,y,\vec z)\to(x+1,y,\vec z)$.

Next we show that $C$ is contractible. Let $w(C)$ denote the weight of
the circuit: that is, the sum of all the coordinates of all the
vertices in $C$. We show that, if $C$ is nontrivial, we can apply one
of the above operations and reduce $w$. Indeed,\vadjust{\goodbreak} let $v$ be a vertex on
$C$ with maximal coordinate sum, and suppose that $v_-$ and $v_+$ are
the vertices before and after $v$ on the circuit. If $v_-=v_+$ then we
can apply the first operation removing $v$ and $v_+$ from the circuit
which obviously reduces $w$. If not, then both $v_-$ and $v_+$ have
strictly smaller coordinate sums than $v$, and we can apply the second
operation reducing $w$ by two. We repeat the above until we reach the
trivial circuit.

Now, let $J$ be the number of edges of $C$ with an end in each
of $E$ and~$I$. The first operation obviously does not change the
parity of $J$. A~simple finite check yields the same for the second
operation. Indeed, assume that we are changing the path from
$(x,y+1),(x+1,y+1),(x+1,y)$ to $(x,y+1),(x,y),(x+1,y)$. Let $F$ be the
set of these four vertices. If no vertex of $I$ is in $F$, then
obviously $J$ does not change. If there is a vertex of $I$ in $F$,
then, by the definition of diagonally connected, $F\cap I=F\cap
\partial_1 E$. Hence the parity of $J$ does not change. [It is even if
$(x,y+1)$ and $(x+1,y)$ are both in $E$ or both in $E^c$ and odd
otherwise.]\vspace*{-2pt}
\end{pf}

Suppose that there is some vertex $v \in\partial_1 E\setminus I$ and
that $u\in E$ is a neighbor of~$v$. Let $y\in I$ and $x\in E$ be
neighbors. Since $E$ and $E^c$ are connected we can find paths
$P_{xu}$ and $P_{vy}$ in $E$ and $E^c$, respectively. The circuit
$P_{xu},uv,P_{vy},yx$ contains a single edge from $E$ to $I$ which
contradicts the claim.\vspace*{-2pt}
\end{pf}

To complete this stage observe that Corollary~\ref{c:sea} holds as
before.\vspace*{-2pt}

\subsection*{Stage 4: Dealing with the difficult subcubes, and Stage 5:
Using the subcubes in the sea to join everything together}

These two stages\vspace*{1pt} go through exactly as before [with one trivial
change: replace $(2c+1)^2$ by $(2c+1)^d$].
This completes the proof of Theorem~\ref{t:multi-dim-gilbert}.\vspace*{-2pt}

%s4 ###
\section{\texorpdfstring{Proof of Theorem \protect\ref{t:knear}}{Proof of Theorem 3}}
In this section we prove Theorem~\ref{t:knear}. Once again, the proof
is very similar to that in Section~\ref{s:gilbert-2d}. We shall
outline the key differences, and emphasise why we are only able to
prove the weaker version of the result.\vspace*{-2pt}

\subsection*{Stage 1: Tessellation}

The tessellation is similar to before, but this time some edges may be
much longer than some nonedges.

Let $k=\cH(G$ is $\kappa$-connected) be the smallest $k$ that
$G_{n,k}$ is $\kappa$-connected. Since $G$ is connected we may assume
that $0.3\log n<k<0.52\log n$ (see~\cite{BBSW1} and~\cite{BBSW2}). Let
$r_-$ be such
that any two points at distance $r_-$ are joined whp; for example,
Lemma 8
of~\cite{BBSW1} implies that this is true provided $\pi r_-^2\le
0.3e^{-1-1/0.3}\log n$, so we can take $r_-= 0.035\sqrt{\log n}$.

Let $r_+$ be such that no edge in the graph has length more than
$r_+$. Then, again by Lemma 8 of~\cite{BBSW1}, we have
\[
\pi r_+^2\le4e(1+0.52)\log n
\]
whp, so we can take $r_+=2.3\sqrt{\log n}\le66r_-$.

From here on, we ignore all point sets with an edge longer than $r_+$
or a nonedge shorter than $r_-$.\vadjust{\goodbreak}

Let $s=r_-/\sqrt8$. We tessellate the box $S_n$ with small squares
of side length~$s$. (Since we are proving only this weaker result our
tesselation does not need to be very fine.) By the choice of $s$ and
the bound on $r_-$ any two points in neighboring or diagonally
neighbouring squares are joined in $G$. Also, by the bound on $r_+$ no
two points in squares with centers at distance more than $(66 \sqrt
5+2)s<150s$ are joined. Let $D=10^4$; we have that no two points in
squares with centers distance $Ds$ apart are joined.\vspace*{-2pt}

\subsection*{Stage 2: The ``difficult'' subsquares}

We call a square \textit{full} if it contains at least $M=10^9$ points
and \textit{nonfull} otherwise. We say two nonfull squares are joined
if they are at $\ell_\infty$ distance at most $2D-1$.

First we bound the size of the largest component of nonfull squares.\vspace*{-2pt}
\begin{lemma}\label{l:non-full-knear}
The largest component of nonfull squares has size less than 7000 whp.\vspace*{-2pt}
\end{lemma}
\begin{pf}
The number of connected subgraphs of $\widehat G$ of size $7000$
containing a particular square is at most $(e(4D)^2)^{7000}$, so, since
there are less than $n$ squares, the total number of such connected
subgraphs is at most $n(e(4D)^2)^{7000}$. The
probability that a square is nonfull is at most
$2s^{2M}e^{-s^2}/M!$. Hence, the expected number of components of
nonfull squares of size at least 7000 is at most
\begin{eqnarray*}
&&n\bigl(2s^{2M}e^{-s^2}(e(4D)^2)/M!\bigr)^{7000}\\
&&\qquad\le n\biggl(2\biggl(\frac{(0.035)^2\log
n}{8}\biggr)^M\frac{e(4D)^2}{M!}\biggr)^{7000}\exp\biggl(\frac{-7000(0.035)^2\log
n}{8}\biggr),
\end{eqnarray*}
which tends to zero as $n$ tends to infinity [since
$7000(-0.035)^2/8>1.07>1$]; that is, whp, no such component exists.\vspace*{-2pt}
\end{pf}

In the rest of the argument we shall assume that there is no nonfull
component of size greater than 7000.\vspace*{-2pt}

\subsection*{Stage 3: The structure of the difficult subsquares}

As usual we fix one component $N$ of the nonfull squares, and suppose
that it has size $u$ (so we know $u<7000$). This time we define
$\widehat
G$ to be the graph on the small squares where each square is joined to
its eight nearest neighbors (i.e., adjacent and diagonal). Let $A=A(N)$
be the giant component of $G\setminus N$, and again split the cutoff
squares into close and far depending whether they have a neighbor (in
$\widehat G$) in $A$.

%f3 ###
%
\begin{figure}

\includegraphics{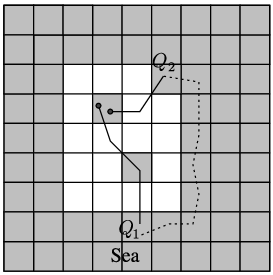}

\caption{Two paths from one cutoff square to the sea together with the
path from the meeting point in $Q_2$ to the square $Q_1$.}
\label{fig:full2}
\end{figure}

By the vertex isoperimetric inequality in the square there are at most
$u^2/2$ squares in $A^c\setminus N$ so $|A^c|\le u^2/2+u<2.5\cdot10^7$.

Next we prove a result similar to Corollary~\ref{c:cutoff}.\vspace*{-2pt}
\begin{lemma}
The set of cutoff squares $A^c$ is in $N_{D}$ (where $D=10^4$ as
above).\vadjust{\goodbreak}
\end{lemma}
\begin{pf}
Suppose not, and that $Q$ is a square in $A^c$ not in $N_{D}$. Then
all squares within $\ell_\infty$ distance of $Q$ at most $D$ are not
in $N$. Hence they must be in $A^c$ (since otherwise there would be
a path from $Q$ to a square in $A$ not going through any square in
$N$). Hence $|A^c|>D^2=10^8$ which
contradicts Lemma~\ref{l:non-full-knear}.
\end{pf}

Finally, we need the analogue of Lemma~\ref{l:connected-boundary}
whose proof is exactly the same.
\begin{lemma}\label{l:connected-boundary-knear}
The set $N_{D}\cap A$ is connected in $\widehat G$.
\end{lemma}

\subsection*{Stage 4: Dealing with the difficult subsquares}

Let us deal with these cutoff squares now. From each cutoff square
that contains at least two vertices, pick any $2$ vertices, and from
each cutoff square that contains a single vertex pick that vertex with
multiplicity two. We have picked at most $5\cdot10^7$ vertices, so
since $G$ is $\kappa=5\cdot10^7$ connected we can simultaneously find
vertex disjoint paths from each of our picked vertices to vertices in
squares in $A$ (two paths from those vertices that are repeated).

We remark that these are not just single edges; these paths may go
through other cutoff squares.

Call the first point of such a path which is in $A$ a \textit{meeting
point}, and the square containing this point a \textit{meeting} square.

Fix a cutoff square and let $v_1,v_2$ be the two vertices picked above
from this square (let $v_1=v_2$ if the square only contains one
vertex). This cutoff square has two meeting points, say $q_1$ and $q_2$
in subsquares $Q_1$ and $Q_2$, respectively. Since the longest edge is
at\vspace*{1pt} most $r_+$, both $Q_1$ and $Q_2$ are in $N_D$. Since $A\cap N_D$ is
connected in $\widehat G$ we construct a path in the squares in $A\cap
N_D$ from the meeting point in $Q_2$ to a vertex in $Q_1$ using at
most one vertex in each subsquare on the way, and missing all the
other meeting points. This is possible since each full square contains
at least $M=10^9$ vertices.\looseness=-1\vadjust{\goodbreak}

Construct a path starting and finishing in $Q_1$ containing all the
(unused) vertices in this cutoff square by joining together the
following paths:
\begin{enumerate}
\item the path from $q_1$ to $v_1$;
\item a path starting at $v_1$ going round all points in the cutoff
square finishing back at $v_2$ (omit this piece if there is just one
far vertex);
\item the path $v_2$ to $q_2$;
\item the path from $q_2$ through $A\cap N_D$ back to $Q_1$
constructed above.
\end{enumerate}

Do this for every cutoff square. For each cutoff square this
construction uses at most two vertices from any square in $A$.
Moreover, it obviously only touches squares in $N_D$. Since nonfull
squares in distinct components are at distance at least $2D$ the
squares touched by different nonfull components are distinct. Thus
in total we have used at most $4\cdot10^7$ vertices in any square in
the sea, and since $M=10^9$ there are many (we shall only need 8)
unused vertices left in each full square in the sea.\vspace*{-2pt}

\subsection*{Stage 5: Using the subsquares in the sea to join everything
together}

This is exactly the same as before.\vspace*{-2pt}

%s5 ###
\section{Comments on the $k$-nearest neighbor proof}

We start by giving some reasons why the proof in the $k$-nearest
neighbor model only yields the weaker Theorem~\ref{t:knear}. The
first superficial problem is that we use squares in the tesselation
which are of ``large'' size rather than relatively small as in the
proof of Theorem~\ref{t:2d-gilbert}, (in other words we did not
introduce the constant $c$ when setting $s$ depending on~$r$).

Obviously we could have introduced this constant. The difficulty when
trying to mimic the proof of Theorem~\ref{t:2d-gilbert} is the large
difference between $r_-$ and $r_+$, which corresponds to having a very
large number of squares (many times $\pi c^2$) in our nonfull
component $N$. This means that we cannot easily prove anything similar
to Lemma~\ref{l:far}. Indeed, a~priori, we could have two far squares
with $\pi c^2$ nonfull squares around each of them.

A different way of viewing this difficulty is that, in the $k$-nearest
neighbor model, the graph $\widehat G$ on the small squares does not
approximate the real graph $G$ very well, whereas in the Gilbert model
it is a good approximation. Thus, it is not surprising that we only
prove a weaker result.

This is typical of results about the $k$-nearest neighbor model; the
results tend to be weaker than for the Gilbert model. This is
primarily because the obstructions tend to be more complex; for
example, the obstruction for connectivity in the Gilbert model is the
existence of an isolated vertex. Obviously in the $k$-nearest
neighbor model we never have an isolated vertex; the obstruction must
have at least $k+1$ vertices.\vspace*{-2pt}

\subsection*{Extensions of Theorem \protect\ref{t:knear}}
When proving Theorem~\ref{t:knear} we only used two facts about the
random geometric graph. First, that any two points at distance
$r_-=0.035\sqrt{\log n}$ are joined whp. Secondly,\vadjust{\goodbreak} that the ratio of
$r_+$ (the longest edge) to $r_-$ (the shortest nonedge) was at most
60 whp. Obviously, we could prove the theorem (with different
constants) in any graph with $r_-=\Theta(\sqrt{\log n})$ and $r_+/r_-$
bounded. This includes higher dimensions and different norms and to
different shaped regions instead of $S_n$ (e.g., to disks or
toruses). Indeed, the only place we used the norm was in obtaining the
bounds on $r_+$ and $r_-$ in stage 1 of the proof.

Indeed, it also generalizes to irregular distributions of vertices
provided that the above bounds on $r_-$ and $r_+$ hold. For example,
it holds in the square $S_n$ where the~density of points in the
Poisson Process decrease linearly from 10 to 1 across the square.\vspace*{-2pt}

%s6 ###
\section{Closing remarks and open questions}

A related model where the result does not seem to follow easily from
our methods is the directed version of the $k$-nearest neighbor
graph. As mentioned above, the $k$-nearest neighbor model naturally
gives rise to a directed graph, and we can ask whether this has a
directed Hamilton cycle. Note that this directed model is
significantly different from the undirected. For example, it is likely
(see~\cite{BBSW1}) that the obstruction to directed connectivity
(i.e., the existence of a directed path between any two vertices) is a
single vertex with in-degree zero; obviously this cannot occur in the
undirected case where every vertex has degree at least $k$. In some
other random graph models a sufficient condition for the existence of
a Hamilton cycle (whp) is that there are no vertices of in-degree or
out-degree zero. Of course, in the directed $k$-nearest neighbor
model every vertex has out-degree $k$ so we ask the following
question.\vspace*{-2pt}
\begin{question*}
Let $\vec G=\vec G_{n,k}$ be the directed $k$-nearest neighbor model. Is
\[
\cH(\mbox{$\vec G$ has a Hamilton cycle})=
\cH(\mbox{$\vec G$ has no vertex of in-degree zero})
\]
whp?\vspace*{-2pt}
\end{question*}

It is obvious that the bound on connectivity in the $k$-nearest
neighbor model can be improved, but the key question is ``should it
be two?'' We make the following natural conjecture:\vspace*{-2pt}
\begin{conjecture*}
Suppose that $k=k(n)$ such that the $k$-nearest neighbor graph
$G=G(k,n)$ is a $2$-connected whp. Then, whp, $G$ has a Hamilton
cycle.\vspace*{-2pt}
\end{conjecture*}

\section*{Acknowledgments}
Some of the results published in this paper were obtained in June 2006
at the Institute of Mathematics of the National University of Singapore
during the program ``11 Random Graphs and Real-world Networks.'' J.~Balogh,
B. Bollob\'as and M. Walters are grateful to the Institute for its
hospitality.\vspace*{-2pt}

% imsref loaded by lrinkeviciute, 2010-11-09 12:58:21
%

%
\printaddresses

\end{document}